\documentclass[12pt]{article}
\usepackage[centertags]{amsmath}
\usepackage{amsfonts}
\usepackage{amssymb}
\usepackage{amsthm}
\usepackage{amsmath,mathrsfs}
\usepackage{amsmath,amscd}
\title{\textbf{$\infty$-links $L$, $\infty$-4-manifolds $M_L$ and Kirby categories}}
\author{Renaud Gauthier \footnote{rg.mathematics@gmail.com} \\ \\}
\theoremstyle{definition}
\newtheorem*{acknowledgments}{Acknowledgments}
\newtheorem{Phai}{Definition}[section]
\newtheorem{KirbyEqu}{Definition}[subsection]
\newtheorem{Mod}[KirbyEqu]{Theorem}
\newtheorem{iop}{Definition}[subsection]
\newtheorem{NGamma}[iop]{Example}
\newtheorem{Catinf}[iop]{Example}
\newtheorem{mapofops}[iop]{Definition}
\newtheorem{catfib}[iop]{Definition}
\newtheorem{fibofop}[iop]{Definition}
\newtheorem{coloredop}[iop]{Definition}
\newtheorem{tXEk}{Definition}[subsection]
\newtheorem{XEk}[tXEk]{Proposition}
\newtheorem{CWCx}{Example}[section]
\newtheorem{Int}[CWCx]{Example}
\newtheorem{rgsp}[CWCx]{Example}
\newtheorem{var}[CWCx]{Example}
\newcommand{\beq}{\begin{equation}}
\newcommand{\eeq}{\end{equation}}
\newcommand{\rarr}{\rightarrow}

\newcommand{\StSo}{\Big( S^2 \times S^1 \Big)^{\#}}
\newcommand{\StDt}{\Big( S^2 \times D^2 \Big)^{\natural}}
\newcommand{\CR}{\mathcal{R}}
\newcommand{\CRi}{\mathcal{R}_{\infty}}
\newcommand{\CK}{\mathcal{K}}
\newcommand{\CL}{\mathcal{L}}
\newcommand{\bCL}{\mathbf{\mathcal{L}}}
\newcommand{\CC}{\mathcal{C}}
\newcommand{\CM}{\mathcal{M}}
\newcommand{\fC}{\mathfrak{C}}
\newcommand{\AXECi}{\text{Alg}_{X \mathbb{E}_2}(\text{Cat}_{\infty})}
\begin{document}
\maketitle
\begin{abstract}
We construct what we call a Kirby category, a monoidal category whose morphisms are smooth 4-manifolds, projecting down to another monoidal category whose morphisms are orientable 3-manifolds, the projection being induced by the boundary map on manifolds. We construct a higher categorical generalization of such concepts and introduce the notion of ribbon $\infty$-categories, a generalization of braided monoidal $\infty$-categories (\cite{Lu1}), which gives rise to the concepts of $\infty$-links, $\infty$-4-manifolds as well as the more general notion of walled $\infty$-4-manifolds if one focuses attention on $\infty$-4-manifolds built from gluing thickened sheets on ribbons. These fall into a larger class of constrained $\infty$-4-manifolds whose classical 4-dimensional counterparts are constrained 4-manifolds on which we consider physical theories. We regard pairs of constrained 4-manifolds and Lagrangians densities depicting physical theories defined on such spaces as morphism objects in an enhanced Kirby category, whose objects are regarded as events. We define a universal category $\Lambda$ of all events that we relate to the $\infty$-category of ribbon $\infty$-categories and conclude in part that Lagrangian field theories can be superseded by using $\infty$-categories.
\end{abstract}

\section{Introduction}
This paper grew out of a desire to make sense of such a notion as an $\infty$-link, a higher categorical generalization of the concept of link. Links can be obtained by closing geometric braids, themselves obtained from braids, morphisms in braided monoidal categories (\cite{Y}). Their $\infty$-categorical counterparts are braided monoidal $\infty$-categories (\cite{Lu1}). Developing a theory of $\infty$-links is done having applications to low dimensional topology in mind. Indeed one can glue 2-handles $D^2 \times D^2$ along framed links $L$ in the boundary $S^3 = \partial D^4$ of the 4-ball to yield smooth 4-manifolds $M_L$ (\cite{GS}). The framing in the $\infty$-setting would correspond to having a notion of ribbon $\infty$-category, something we develop in this paper. When we close geometric braids to obtain links we technically identify domain and codomain of their corresponding braid morphisms but in general one may not be able to do such a thing if domain and codomain differ. We generalize the process whereby one obtains 4-manifold by gluing handles $D^2 \times D^2$ along framed links to gluing thickened sheets $D^2 \times \widetilde{D^2}$ ($\widetilde{D^2}$ the blow up of the disk at the origin) along ribbons, giving rise to something we call constrained manifolds. Those 4-manifolds we regard as target spaces of physical theories and think of them as morphism objects between events, objects that have a dynamic flavor to them, and that we define to be objects of a representation theoretic nature. We are led to defining what we call the Kirby category whose objects are thickened events (to allow for the presence of framed links) and whose morphisms are smooth 4-manifolds. The boundary map on manifolds is shown to give rise to a moduli category $\Phi$ whose objects are events and morphisms are oriented 3-manifolds. The Kirby category $\CK$ gives rise under constraints to something called a constrained category $\CK^{\fC}$ whose objects are constrained 4-manifolds, 4-manifolds obtained from smooth 4-manifolds by adding in constraints. Constrained 4-manifolds form a family that contains objects such as 4-manifolds obtained from gluing thickened sheets along ribbons in $S^3$, the constraints here being the blow-up of the second component of $D^2 \times D^2$ to obtain a sheet $D^2 \times \widetilde{D^2}$ as well as the introduction of walls. We enhance $\CK^{\fC}$ by considering physical theories on such spaces, hence the introduction of $\CK^{\CL, \fC}$, the exact same category with the addition that morphisms are now pairs $(\CL, M)$ where $M$ is a constrained 4-manifold and $\CL$ is a Lagrangian density corresponding to a physical theory having $M$ for target space. Regarding the events themselves we subsume the existence of a putative universal category $\Lambda$ of Events whose realization in 3-dimensional space are the events we have been discussing. On the $\infty$-categorical side of the problem, we have the $\infty$-category $\AXECi$ of maps of $\infty$-operads from the twisted little cubes operad $X \mathbb{E}_2$, a generalization of the little k-disks operad $\mathbb{E}_2$ we define, to the $\infty$-category Cat$_{\infty}$ (\cite{Lu2}). Each such map defines a ribbon $\infty$-category whose morphisms $f:A \rarr B$ have an associated $\infty$-groupoid Map$(A,B)_f$ that we regard as an $\infty$-ribbon. If there is an equivalence that identifies domain and codomain of $f$, under such an identification one obtains a framed $\infty$-link Map$(A,B)_{\overline{f}}$. Coupling this with the 4-manifold $M_{\overline{f}}$ one obtains what we call an $\infty$-4-manifold $(M_{\overline{f}}, \text{Map}(A,B)_{\overline{f}})$. Morphisms $f$ for which there is no identification of their domain and codomain give rise to 4-manifolds constructed from gluing $D^2 \times \widetilde{D^2}$ on the ribbon $f$ seen as stretching between two walls, giving rise to what we call walled 4-manifolds $M_f$, and the pairs $(M_f, \text{Map}(A,B)_f)$ are referred to as walled $\infty$-4-manifolds. $M_f$ is manifestly constrained. For a given morphism $f$, we ask the question whether there is a relation between Map$(A,B)_f$ and Lagrangian densities $\CL$ corresponding to physical theories defined on $M_f$. This leads us to comparing $\Lambda$ and $\AXECi$. This paper presents the machinery necessary for making such a comparison.\\

In section \ref{section1} we motivate the necessity of considering events and their original counterparts, Events, objects in an abstract category $\Lambda$ all of whose objects are Events. We indicate that events occur in 3-dimensional space and that gives rise to the monoidal category $\Phi$ with events as objects and whose morphisms are oriented 3-manifolds. In section \ref{section2} we define the Kirby category $\CK$ a monoidal category whose objects are thickened events and whose morphisms are oriented smoothable 4-manifolds. Thickening events requires to reconsider the handle slides, something we tackle first. We then show that $\Phi$ appears as a moduli category with respect to $\CK$ via the use of a categorical generalization of the boundary functor $\partial$ on manifolds. In section \ref{section3} we present the importance of having an $S^1$-action on some structure intrinsic to the 4-manifolds that we will regard as target spaces of physical theories, such an action being seen as illustrative of an inherent dynamical flavor that allows one to define dynamical objects lying on such spaces. Such an action however is not present in general at the moment and that leads us to lower our goals to defining constrained 4-manifolds, obtained from smooth 4-manifolds by adding in some constraints, spaces that may not display an intrinsic $S^1$-action but which nevertheless happen to be target spaces of some physical theories. In section \ref{section4} we define the ribbon $\infty$-category $\CRi$ of thickened events by first going fairly fast over $\infty$-operads and by defining the twisted little k-disks operad $X \mathbb{E}_2$. The $\infty$-category $\AXECi$ has for objects ribbon $\infty$-categories, one of which, $\CRi$, has for objects thickened events and will give rise to $\infty$-links and $\infty$-4-manifolds as well as generalizations of such concepts, something we call walled $\infty$-braids and walled $\infty$-4-manifolds. In the last section we put things in perspective and try to address the connection between $\Lambda$ and $\AXECi$.

\begin{acknowledgments}
The author would like to express his gratitude to D. Gaitsgory for answering a key question about D-modules and derived algebraic stacks.
\end{acknowledgments}

\section{3-manifolds as Hom objects} \label{section1}
\subsection{Need for an ambient space - Motivation}
We call \textbf{data} representation theoretic objects such as, but not limited to, fiber functors from Tannakian $\infty$-categories to $\infty$-categories of rigid modules for instance (\cite{W}). We will limit ourselves to saying that data are objects of a representation theoretic nature. We call \textbf{Events} a minimal collection of data representing the same object which is complete in the sense that addition of any other datum representing that same object to such a collection makes such an action redundant, and omission of one piece of data makes the collection possibly non-unique in that one could separately add two different pieces of data to such a collection resulting in two distinct collections of data, those additional two data representing the same object as that being represented by the collection. We call \textbf{events} the realization of Events in 3-dimensional space. This can be implemented by seeing Events as objects of a category $\Lambda$, on which we put a Grothendieck topology, and by considering a sheaf of events over this site which we define to be a category $\Phi$ whose objects are events. $\Phi$ is defined in \ref{fai} and $\Lambda$ will be defined below. We make the fundamental assumption that there is an essentially surjective functor from $\Lambda$ to any category having objects constructed from events.\\

We are ultimately interested in having distinct events and using functors lends itself well to this setting. The (possibly) continuous process whereby one goes from one event to the next is encoded in the Hom objects in $\Phi$, those are sequences of orientable 3-manifolds. This makes it somewhat manifest that time for us will be taken as being discrete. The indexing of Events will be sufficient to determine at which point we are situated in time. For instance given a sequence $\{E_i\}_{i \geq 1}$, for $E_n$ a fixed Event, Events $E_p$ for $p < n$ are past Events, Events $E_q$ for $q > n$ are yet to come.\\

\subsection{$\Lambda$, the category of Events}

By definition Events are fully known. We claimed earlier that one can assemble $\Lambda$ into a category. Concretely this means taking Events as objects, morphisms from one event $E$ to another event $E'$ are any map $E \rarr E'$, the identity is defined object-wise: $id|_E = id_E$, composition is defined by just being a succession of morphisms and for $f:E \rarr E'$ a morphism from one event $E$ to another event $E'$, it is immediate that $id \circ f = id|_{E'} \circ f = id_{E'} \circ f = f$ as well as $f \circ id = f \circ id|_E = f \circ id_E = f$. Composition of morphisms being defined as a succession of maps immediately yields a notion of grading on morphisms which we define for a given morphism to be the number of maps involved in defining that particular morphism. In particular if $f:E \rarr E'$ does not result from composing other maps then it is defined to have degree one. The following morphism:
\beq
E_1 \xrightarrow{f_1} E_2 \rarr \cdots E_{n-1} \xrightarrow{f_{n-1}} E_n
\eeq
being a succession of $n-1$ maps, none of which is itself a succession of other maps, is defined to have degree $n-1$. Let $\Lambda^{\geq n}(E):=\{f:E_1 \rarr \cdots \rarr E_n \rarr E \}$ be the family of morphisms in $\Lambda$ that result from composing at least $n$ morphisms. This defines a sieve over $E$. We take $\{\Lambda^{\geq n}(E) \}_{n \geq 0}$ to be a family of covering sieves. Note that it includes $\Lambda^{\geq 0}(E)$ the set of all morphisms to $E$ in $\Lambda$. Let $\Lambda^{\geq n}(E)$ be a covering sieve, $\psi:E' \rarr E$ a morphism, the restriction $\Lambda_{E'}^{\geq n}(E)$ of $\Lambda^{\geq n}(E)$ to $E'$ is the family of morphism $E'' \rarr E'$ such that $E'' \rarr E' \rarr E$ is in $\Lambda^{\geq n}(E)$, but every such composition is in $\Lambda^{\geq n}(E)$, so $\Lambda_{E'}^{\geq n}(E) = \Lambda^{\geq 1}(E')$ is a covering sieve, and this for all $n\geq 0$. Let now $\Xi$ be a sieve over $E$, $\Lambda^{\geq n}(E)$ a covering sieve for $n$ fixed and positive, such that for all morphisms $f: E' \rarr E$ in $\Lambda^{\geq n}(E)$, $\Xi_{E'} = \{ E'' \rarr E' \: | \: \theta:E'' \rarr E' \rarr E \text{ is in } \Xi \}$ is a covering sieve $\Lambda^{\geq r}(E')$ for some $r$. Note that if a sieve has a morphism $E'' \rarr E$ then it also contains all morphisms from the overcategory $\Lambda_{/E''}$, hence we must have $r=1$, and this for all $f:E' \rarr E$ in $\Lambda^{\geq n}(E)$, hence $\Xi =\Lambda^{\geq n}(E)$ is a covering sieve. This being true for all $E$ it follows that we have a Grothendieck topology on $\Lambda$, making it into a site (\cite{GM}).\\

It is sometimes convenient to adopt an $\mathbb{S}^{\infty}$ approach (\cite{RG}). In that paper, an $\infty$-sphere is simply defined as a collection of infinitely many points each of which is connected to infinitely many other points. In our situation points of $\mathbb{S}^{\infty}[\Lambda]$ are the objects of $\Lambda$, a flow from one point $E$ to another point $E'$ being given by a sequence of morphisms $E \rarr E_1 \rarr \cdots \rarr E_n \rarr E'$ in $\Lambda$. A \textbf{flow} on $\mathbb{S}^{\infty}$ from $E$ to $E'$ is given by the collection of all such sequences and defines Map$_{\mathbb{S}^{\infty}}(E,E')$ that we denote by $\mathbb{S}^{\infty} - E - E'$, which is slightly misleading as not all points of $\mathbb{S}^{\infty} - E -E'$ may find themselves on a flow from $E \rarr E'$, but we find this representation to be a good visual aid. It is not difficult to go back and forth between $\Lambda$ and $\mathbb{S}^{\infty}$.\\

\subsection{$\Phi$  monoidal category of events and 3-manifolds} \label{fai}

We consider the collection of orientable 3-manifolds which we denote by $\CM^3$. We consider directed families of elements of $\CM^3$. Let $\overrightarrow{\CM^3}$ be the collection of such directed families of elements of $\CM^3$.\\

Ultimately 3-manifolds will be obtained as the boundary of 4-manifolds $M_L$ obtained by gluing 2-handles along framed links $L$ in $S^3$. Links are obtained by closing geometric braids, whose corresponding braids are morphisms in braided monoidal categories. Consider a braid $A_1 \otimes \cdots \otimes A_n \xrightarrow{B} A_1 \otimes \cdots \otimes A_n$. Preceding that braid by the identity on all objects we obtain:
\beq
A_1 \otimes \cdots \otimes A_n \xrightarrow{id} A_1 \otimes \cdots \otimes A_n \xrightarrow{B} A_1 \otimes \cdots \otimes A_n \nonumber
\eeq
If we close each individual braid, one obtains a 0-framed unlink with $n$ components followed by $L$. Gluing 2-handles along the unlink one obtains the boundary sum $\natural n S^2 \times D^2$ (\cite{GS}) whose boundary $\partial \natural n S^2 \times D^2 = \# n \partial (S^2 \times D^2)=\# n S^2 \times S^1$. Thus from that perspective, we regard a sequence of braids:
\beq
\otimes^n A_i \xrightarrow{B_1} \otimes^n A_i \cdots \xrightarrow{B_p} \otimes^n A_i \nonumber
\eeq
as being equivalent to having the following sequence:
\beq
\cdots \xrightarrow{id} \otimes^n A_i \xrightarrow{B_1} \otimes^n A_i \cdots \xrightarrow{B_p} \otimes^n A_i \xrightarrow{id} \otimes^n A_i \xrightarrow{id} \cdots \nonumber
\eeq
where identity braids extend to the left and to the right. Upon closing each individual corresponding geometric braid, one arrives at a finite sequence of framed links $L_1, \cdots, L_p$ in between sequences of unlinks on $n$ components. Surgery on those links yields a finite sequence of 4-dimensional 2-handlebodies $M_{L_1}, \cdots, M_{L_p}$ in between infinite sequences of $\natural n S^2 \times D^2$ which from our perspective are considered to be trivial, hence we regard a directed family such as:
\beq
\cdots \rarr \natural n S^2 \times D^2 \rarr M_{L_1} \rarr \cdots \rarr M_{L_p} \rarr \natural n S^2 \times D^2  \rarr \cdots \nonumber
\eeq
as really being equivalent to:
\beq
 M_{L_1} \rarr \cdots \rarr M_{L_p}
\eeq
Upon taking the individual boundaries of such directed families of 4-manifolds one respectively arrives at (after letting $M_i = M_{L_i}$):
\beq
\cdots \rarr \# n S^2 \times S^1 \rarr M_1 \rarr \cdots \rarr M_p \rarr \# n S^2 \times S^1  \rarr \cdots \nonumber
\eeq
and:
\beq
 M_1 \rarr \cdots \rarr M_p
\eeq
both of which we regard as being equivalent. $n$ will of course depend on the original morphisms in the braided monoidal category we started with. All such sequences however can be completed by such a boundary sum of $S^2 \times D^2$ in the case of 4-manifolds, and connected sums of sphere products $S^2 \times S^1$ in the case of 3-manifolds. Focusing our attention on 3-manifolds denote by $\StSo$ a variable whose domain is the collection of all such connected sums, which we regard as an identity for the composition of elements of $\overrightarrow{\CM^3}$.\\

Thus we define a directed family of orientable 3-manifolds $M_1 \rarr \cdots \rarr M_n$ to be equivalently given by completing this family by $\StSo$ to the right and to the left as in:
\beq
\cdots \rarr \StSo \rarr \StSo \rarr M_1 \rarr \cdots \rarr M_n \rarr \StSo \rarr \cdots \nonumber
\eeq
\begin{Phai}
We let $\Phi$ be the formal monoidal category whose objects are events and Hom objects are elements of $\overrightarrow{\CM^3}$.
\end{Phai}

In a first time we address the use of the adjective formal. If some events of $\Phi$ can be collected together and assembled into a genuine monoidal category, this is not necessarily true if we enlarge such a collection to all events. Thus for events belonging to distinct monoidal sub-categories we nevertheless enforce a tensor product regarded as a formal such operation. This means also that we have a formal unit object $I$ on $\Phi$, much in the same spirit as when defining $\StSo$. For $\CC$ and $\CC'$ two distinct monoidal categories, $e_{\CC}$, $e_{\CC,1}$ and $e_{\CC,2}$ objects of $\CC$, $e_{\CC'}$ an object of $\CC'$, $1_{\CC}$ the unit object of $\CC$, $\otimes$ the tensor product on $\Phi$, then:
\begin{align}
e_{\CC,1} \otimes e_{\CC,2}&:=e_{\CC,1} \otimes_{\CC} e_{\CC,2} \\
e_{\CC} \otimes I &= e_{\CC} \otimes_{\CC} 1_{\CC}  \\
I \otimes e_{\CC} &= 1_{\CC} \otimes_{\CC} e_{\CC}
\end{align}
We regard $e_{\CC} \otimes e_{\CC'}$ as being formal. We also have:
\begin{align}
\big( e_{\CC,1} \otimes e_{\CC'} \big) \otimes I &=  e_{\CC,1} \otimes \big( e_{\CC'}  \otimes I \big) \\
&=  e_{\CC,1} \otimes \big( e_{\CC'}  \otimes_{\CC'} 1_{\CC'} \big) \\
&= e_{\CC,1} \otimes e_{\CC'}
\end{align}

It is assumed that the orientable 3-manifolds in this definition are known at the onset and can be recovered from viewing $\Phi$ as a sheaf over $\Lambda$. Let $\Phi(x,y)$ be the set of morphisms between objects $x$ and $y$ of $\Phi$. An element of $\Phi(x,y)$ is a directed family, element of $\overrightarrow{\CM^3}$. Composition in $\Phi$ is obtained by concatenating directed families. For any triple $(x, y, z)$ of objects of $\Phi$, we define the composition:
\beq
\mu_{x,y,z}: \Phi(x,y) \otimes \Phi(y,z) \rarr \Phi(x,z)
\eeq
to be induced by the concatenation:
\begin{align}
(M_1 \rarr &\cdots \rarr M_n ) \times (N_1 \rarr \cdots \rarr N_p ) \nonumber \\
&\rarr (M_1 \rarr \cdots \rarr M_n \rarr N_1 \rarr \cdots \rarr N_p )
\end{align}
where $M_1 \rarr \cdots \rarr M_n $ is an element of $\Phi(x,y)$, $N_1 \rarr \cdots \rarr N_p  $ is an element of $\Phi(y,z)$ and the directed family above is obtained by concatenating those two directed families, which we can denote by $\{M_r\} \rarr \{N_s \}$. The identity for such a composition is taken to be $\StSo$ as in:
\begin{align}
 \StSo  \times (M_1 \rarr \cdots \rarr M_n ) &\rarr ( \StSo \rarr M_1 \rarr \cdots \rarr M_n ) \nonumber \\
  &\rarr ( \#p S^2 \times S^1 \rarr M_1 \rarr \cdots \rarr M_n ) \nonumber \\
 \sim  \cdots  \rarr \#p S^2 \times S^1 \rarr M_1 &\rarr \cdots \rarr M_n \rarr \#p S^2 \times S^1 \rarr \cdots  \nonumber \\
& \sim M_1 \rarr \cdots \rarr M_n
\end{align}
where $x=\otimes^p e_i$ and likewise $(M_1 \rarr \cdots \rarr M_n ) \times  \StSo  \rarr (M_1 \rarr \cdots \rarr M_n )$.\\
Associativity of the composition is obvious by concatenation:
\beq
\begin{CD}
\{M_n \} \times \{ N_m \} \times \{P_q \} @>>> \Big(\{M_n \} \rarr \{ N_m \}\Big) \times \{P_q \} \\
@VVV @VVV \\
\{M_n \} \times \Big( \{ N_m \} \rarr \{P_q \}\Big) @>>> \Big(\{M_n \} \rarr \{ N_m \} \rarr \{P_q \}\Big)
\end{CD}
\eeq

\section{Kirby categories} \label{section2}

\subsection{Kirby moves}
For completeness' sake, we show what the Kirby moves are at the level of tangles. Let $\CR$ be a ribbon category.\\

The blow up of an unknotted circle of framing $\pm 1$ is performed following the sequence below (we take an unknot with framing $+1$ for illustrative purposes). We start with a braid $B$:
\beq
\setlength{\unitlength}{0.5cm}
\begin{picture}(3,3)(0,0)
\multiput(0,3)(0.3,0){10}{\line(1,0){0.15}}
\put(1,1){$B$}
\multiput(0,0)(0.3,0){10}{\line(1,0){0.15}}
\end{picture}
\nonumber
\eeq
the dotted lines showing where the objects should be situated, $B$ being a morphism in $\CR$ from the object on top to the one below. Take the tensor product of those with the unit object $\mathbf{1}$:
\beq
\setlength{\unitlength}{0.5cm}
\begin{picture}(7,3)(0,0)
\multiput(0,3)(0.3,0){10}{\line(1,0){0.15}}
\put(3,3){$\otimes$}
\put(5,3){$\mathbf{1}$}
\put(1,1){$B$}
\multiput(0,0)(0.3,0){10}{\line(1,0){0.15}}
\put(3,0){$\otimes$}
\put(5,0){$\mathbf{1}$}
\end{picture}
\nonumber
\eeq
Stretch the braid $B$ down using trivial strands $id$:
\beq
\setlength{\unitlength}{0.5cm}
\begin{picture}(7,7)(0,0)
\multiput(0,6)(0.3,0){10}{\line(1,0){0.15}}
\put(3,6){$\otimes$}
\put(5,6){$\mathbf{1}$}
\put(1,4){$B$}
\multiput(0,3)(0.3,0){10}{\line(1,0){0.15}}
\put(3,3){$\otimes$}
\put(5,3){$\mathbf{1}$}
\put(1,1){$id$}
\multiput(0,0)(0.3,0){10}{\line(1,0){0.15}}
\put(3,0){$\otimes$}
\put(5,0){$\mathbf{1}$}
\end{picture}
\nonumber
\eeq
where in the lower band we can now blow up the framed unknot:
\beq
\setlength{\unitlength}{0.5cm}
\begin{picture}(7,7)(0,0)
\multiput(0,6)(0.3,0){10}{\line(1,0){0.15}}
\put(3,6){$\otimes$}
\put(5,6){$\mathbf{1}$}
\put(1,4){$B$}
\multiput(0,3)(0.3,0){10}{\line(1,0){0.15}}
\put(3,3){$\otimes$}
\put(5,3){$\mathbf{1}$}
\put(1,1){$id$}
\multiput(0,0)(0.3,0){10}{\line(1,0){0.15}}
\put(3,0){$\otimes$}
\put(5,0){$\mathbf{1}$}
\put(5,2){\oval(1,0.5)[t]}
\put(4.5,2){\line(1,-1){1}}
\qbezier(4.5,1)(4.5,1.5)(4.7,1.5)
\qbezier(5.5,2)(5.5,1.5)(5.3,1.5)
\put(5,1){\oval(1,0.5)[b]}
\end{picture}
\nonumber
\eeq
For the blow-down of the unknot, if the unknot is right next to the braid, pull it down in a band so that we end up in a situation like the last one above. Reverse the blow-up process above from there.\\

For the band sum move, we have two options: addition or subtraction. Whether it be an addition or a subtraction, we first have to introduce a copy of the link component being band summed over following the framing of the link under consideration. At the level of braids, without loss of generality, this means considering braids in the blackboard framing. The copy is effectively produced by something we call a \textbf{shadow map} $\delta$ that creates an unlabeled copy of an object directly to its left. We denote by $A_{\ast}$ this copy, a variable in Ob$(\CR)=\{A_i\}_{i \in I}$, $I$ an indexing set, hence $A_{\ast} =$ MotFr$_0(\CR)$ and we will simply refer to $A_{\ast}$ as a shadow object (see Appendix for details). Then:
\begin{align}
\delta: \text{Ob}(\CR) &\rarr \text{MotFr}_0 \otimes \text{Ob}(\CR) \nonumber \\
A &\mapsto A_{\ast} \otimes A
\end{align}

For one strand corresponding to a portion of the link component being summed over, we have to bring another strand corresponding to a portion of the link doing the band sum move close to that first strand. Crossings (or twists $\sigma$) may be necessary to do just this. Then we find ourselves in either of two situations. An addition corresponds to having
the shadow object having the same sign as the object it's originating from. Then we consider local windows on the braids such as in:
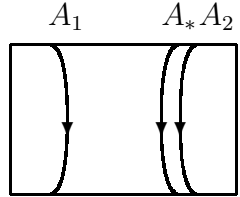
\beq
\setlength{\unitlength}{0.5cm}
\begin{picture}(6,5)(0,0)
\thicklines
\put(0,0){\line(1,0){6}}
\put(0,0){\line(0,1){4}}
\put(6,0){\line(0,1){4}}
\put(0,4){\line(1,0){6}}
\put(1,4.5){$A_1$}
\put(4,4.5){$A_{\ast}$}
\put(5,4.5){$A_2$}
\qbezier(1,4)(1.5,4)(1.5,2)
\qbezier(1.5,2)(1.5,0)(1,0)
\put(1.5,2){\vector(0,-1){0.5}}
\qbezier(4.5,4)(4,4)(4,2)
\qbezier(4,2)(4,0)(4.5,0)
\put(4,2){\vector(0,-1){0.5}}
\qbezier(5,4)(4.5,4)(4.5,2)
\qbezier(4.5,2)(4.5,0)(5,0)
\put(4.5,2){\vector(0,-1){0.5}}
\end{picture}
\nonumber
\eeq
where $\delta A_2= A_{\ast} \otimes A_2$. An addition can simply be done by replacing the identities on $A_1$ and $A_{\ast}$ by the twist map $\sigma: A_1 \otimes (A_{\ast}=A_1) \rarr (A_{\ast}=A_1) \otimes A_1$ as in:
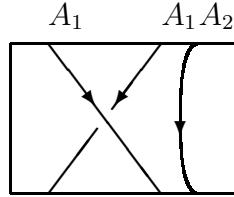
\beq
\setlength{\unitlength}{0.5cm}
\begin{picture}(6,5)(0,0)
\thicklines
\put(0,0){\line(1,0){6}}
\put(0,0){\line(0,1){4}}
\put(6,0){\line(0,1){4}}
\put(0,4){\line(1,0){6}}
\put(1,4.5){$A_1$}
\put(4,4.5){$A_1$}
\put(5,4.5){$A_2$}
\put(1,4){\line(3,-4){3}}
\put(1.75,3){\vector(3,-4){0.5}}
\put(1,0){\line(3,4){1.3}}
\put(4,4){\line(-3,-4){1.3}}
\put(3.25,3){\vector(-3,-4){0.5}}
\qbezier(5,4)(4.5,4)(4.5,2)
\qbezier(4.5,2)(4.5,0)(5,0)
\put(4.5,2){\vector(0,-1){0.5}}
\end{picture}
\nonumber
\eeq
Once the transition is made from identity strands on $A_1$ and $A_{\ast}$ to a morphism that involves both objects, the shadow object $A_{\ast}$ takes on the identity of $A_1$ as indicated in parentheses above, or whatever object it becomes related to, which corresponds to a specialization of the variable $A_{\ast}$.\\

A subtraction occurs when $A_{\ast}$ and $A_2$ have opposite signs, in which case we find ourselves in the following situation:
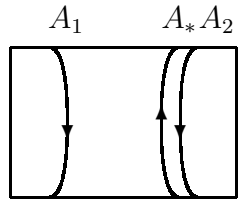
\beq
\setlength{\unitlength}{0.5cm}
\begin{picture}(6,5)(0,0)
\thicklines
\put(0,0){\line(1,0){6}}
\put(0,0){\line(0,1){4}}
\put(6,0){\line(0,1){4}}
\put(0,4){\line(1,0){6}}
\put(1,4.5){$A_1$}
\put(4,4.5){$A_{\ast}$}
\put(5,4.5){$A_2$}
\qbezier(1,4)(1.5,4)(1.5,2)
\qbezier(1.5,2)(1.5,0)(1,0)
\put(1.5,2){\vector(0,-1){0.5}}
\qbezier(4.5,4)(4,4)(4,2)
\qbezier(4,2)(4,0)(4.5,0)
\put(4,2){\vector(0,1){0.5}}
\qbezier(5,4)(4.5,4)(4.5,2)
\qbezier(4.5,2)(4.5,0)(5,0)
\put(4.5,2){\vector(0,-1){0.5}}
\end{picture}
\nonumber
\eeq
and the subtraction can simply be done by replacing the identity maps on $A_1$ and $A_{\ast}$ by the concatenation of the maps $\varepsilon:  A_1 \otimes A_1^* \rarr \mathbf{1}$ and $\eta: \mathbf{1} \rarr A_1 \otimes A^{\ast}_1$ where $A_{\ast}$ has been specialized to $A_1 ^{\ast}$, as in:
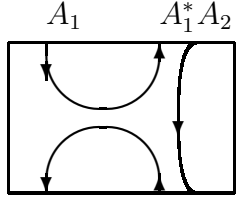
\beq
\setlength{\unitlength}{0.5cm}
\begin{picture}(6,5)(0,0)
\thicklines
\put(0,0){\line(1,0){6}}
\put(0,0){\line(0,1){4}}
\put(6,0){\line(0,1){4}}
\put(0,4){\line(1,0){6}}
\put(1,4.5){$A_1$}
\put(4,4.5){$A_1^*$}
\put(5,4.5){$A_2$}
\put(2.5,4){\oval(3,3.5)[b]}
\put(1,3.5){\vector(0,-1){0.5}}
\put(4,3.5){\vector(0,1){0.5}}
\put(2.5,0){\oval(3,3.5)[t]}
\put(1,0.5){\vector(0,-1){0.5}}
\put(4,0){\vector(0,1){0.5}}
\qbezier(5,4)(4.5,4)(4.5,2)
\qbezier(4.5,2)(4.5,0)(5,0)
\put(4.5,2){\vector(0,-1){0.5}}
\end{picture}
\nonumber
\eeq

\subsection{The Kirby category $\CK$}

We define a formal topological monoidal category that we call the \textbf{Kirby category} $\CK$ as follows: its objects are \textbf{thickened events} $\delta e_i$, viewed as objects in a certain braided monoidal category with a twist map, each object of which has a right dual, wich is also known as a ribbon category. The monoidal structure on thickened events is the generalization to the thickened setting of the monoidal structure we had on $\Phi$, the unit object for such a structure being given by the thickened unit:
\beq
\delta I_{\Phi}=:I_{\CK}
\eeq
Regarding the definition of morphisms in $\CK$, recall that any morphism between any two same objects of a ribbon category can be represented by a geometric braid whose closure gives a framed link. We regard this link that we denote by $L$ as being in the boundary $S^3= \partial D^4$ of the 4-ball. Upon gluing 2-handles along such a framed link we obtain a 4-manifold $M_L$ (\cite{GS}). To fix notations, if $\CR$ is a ribbon category, $A_1$, $A_2, \cdots$, $A_n$ are objects of $\CR$, $B$ is a morphism from $A_1 \otimes \cdots \otimes A_n$ to itself, $L$ the closure of $B$, then to such objects and braid $B$ we associate the object $A_1 \otimes \cdots \otimes A_n$ of some category pre$\CK(\CR)$, with $M_L$ the following morphism:
\beq
A_1 \otimes \cdots \otimes A_n \xrightarrow{M_L} A_1 \otimes \cdots \otimes A_n
\nonumber
\eeq
As done in the previous section, recall that extending our braid using trivial strands, or equivalently the identity on objects in $\CR$, one arrives at:
\beq
\cdots \xrightarrow{id} A_1 \otimes \cdots \otimes A_n \xrightarrow{B} A_1 \otimes \cdots \otimes A_n
\xrightarrow{id} \cdots \nonumber
\eeq
which upon closure of each geometric braid representing the above braids yields:
\beq
\cdots \xrightarrow{nS^1} A_1 \otimes \cdots \otimes A_n \xrightarrow{L} A_1 \otimes \cdots \otimes A_n
\xrightarrow{nS^1} \cdots \nonumber
\eeq
where $nS^1$ is short for 0-framed unlink with $n$ components. Surgery along those individuals links yields:
\beq
\cdots \xrightarrow{\natural n S^2 \times D^2} A_1 \otimes \cdots \otimes A_n \xrightarrow{M_L} A_1 \otimes \cdots \otimes A_n \xrightarrow{\natural n S^2 \times D^2} \cdots \nonumber
\eeq
and in the same manner that we defined $\StSo$ in the previous section we are led to defining $\StDt$ as a variable that represents boundary sums $\natural p S^2 \times D^2$ and which we regard as an identity. This means that we have:
\begin{align}
 \StDt  \times (M_{L_1} \rarr \cdots \rarr M_{L_n} ) &\rarr ( \StDt \rarr M_{L_1} \rarr \cdots \rarr M_{L_n} ) \nonumber \\
  &\rarr ( \natural p S^2 \times D^2 \rarr M_{L_1} \rarr \cdots \rarr M_{L_n} ) \nonumber \\
 \sim  \cdots \rarr \natural p S^2 \times D^2 \rarr M_{L_1} &\rarr \cdots \rarr M_{L_n} \rarr \natural p S^2 \times D^2 \rarr \cdots  \nonumber \\
& \sim M_{L_1} \rarr \cdots \rarr M_{L_n}
\end{align}
and likewise $(M_{L_1} \rarr \cdots \rarr M_{L_n} ) \times  \StDt  \rarr (M_{L_1} \rarr \cdots \rarr M_{L_n} )$.\\

This can be formalized as follows. One can view the procedure of constructing 4-manifolds from links as a map $\phi: \Omega S^3 \rarr \CM^4, L \mapsto M_L$, where $\CM^4$ denotes the collection of all 4-manifolds. If we further denote by $\overrightarrow{\Omega S^3}$ the collection of directed families of links in $\Omega S^3$, and by $\overrightarrow{\CM^4}$ the collection of directed families of elements of $\CM^4$, then one can extend the surgery map $\phi$ to:
\begin{align}
\phi: \overrightarrow{\Omega S^3} &\rarr \overrightarrow{\CM^4} \nonumber \\
 (L_1 &\rarr \cdots \rarr  L_n) \nonumber \\
& \mapsto  (M_{L_1} \rarr \cdots \rarr  M_{L_n})
\end{align}
This can be further generalized by defining a category $\overline{\CR}$ whose objects are the same as those of $\CR$ but whose morphisms are directed families of links $L=\overline{B}$, $B \in \text{Mor}(\CR)$, and by defining a functor $\overline{\CR} \rarr \text{pre}\CK(\CR)$ that is the identity on objects and is $\phi$ on Hom objects.\\

Filtered families of morphisms in $\CR$:
\beq
\Bigg\{  \begin{CD} A_1 \otimes \cdots \otimes A_n \\
                         @VV B^{(1)}V \\
                          A_1 \otimes \cdots \otimes A_n \end{CD} ,
           \begin{CD} A_1 \otimes \cdots \otimes A_n \\
                         @VV B^{(2)}V \\
                          A_1 \otimes \cdots \otimes A_n \end{CD} ,                          \cdots
            \begin{CD} A_1 \otimes \cdots \otimes A_n \\
                         @VV B^{(p)}V \\
                          A_1 \otimes \cdots \otimes A_n \end{CD}\Bigg\}
\eeq
correspond to directed families:
\beq
\begin{CD} A_1 \otimes \cdots \otimes A_n \\
                         @VV B^{(1)}V \\
                          A_1 \otimes \cdots \otimes A_n \end{CD} \rarr
          \begin{CD} A_1 \otimes \cdots \otimes A_n \\
                         @VV B^{(2)}V \\
                          A_1 \otimes \cdots \otimes A_n \end{CD} \rarr
            \cdots \rarr
            \begin{CD} A_1 \otimes \cdots \otimes A_n \\
                         @VV B^{(p)}V \\
                          A_1 \otimes \cdots \otimes A_n \end{CD}
\eeq
which upon closure correspond to the directed family $L^{(1)} \rarr \cdots \rarr L^{(p)}$ of links, $L^{(i)} = \overline{B^{(i)}}$, a morphism in $\overline{\CR}$ :
\beq
A_1 \otimes \cdots \otimes A_n \xrightarrow{L^{(1)} \rarr \cdots \rarr L^{(p)}} A_1 \otimes \cdots \otimes A_n
\eeq
which under the surgery functor $\phi$ maps to the following morphism in pre$\CK(\CR)$:
\beq
A_1 \otimes \cdots \otimes A_n \xrightarrow{ M_{L^{(1)}} \rarr \cdots \rarr M_{L^{(p)}}} A_1 \otimes \cdots \otimes A_n
\eeq
We define $\CK$ to be the essential image pre$\CK(\CR)$ of $\overline{\CR}$ under $\phi$ where $\CR$ has for objects thickened events $\delta e_i$'s.

\subsection{Moduli categories}
It is known (\cite{L}) that two 4-manifolds $M_L$ and $M_{L'}$ have the same boundary if and only if $L$ and $L'$ are related by a sequence of Kirby moves, that is blow-ups and blow-downs of unknots with framing $\pm 1$ and band sum moves (equ. handle slides) (\cite{R}). This motivates the following definition:
\begin{KirbyEqu}
Two 4-manifolds $M_L$ and $M_{L'}$ are said to be \textbf{Kirby equivalent} if and only if $L$ and $L'$ are related by a sequence of Kirby moves.
\end{KirbyEqu}
This relation is clearly symmetric and transitive. Reflexivity follows by, for example, blowing up an unknotted circle with framing $\pm 1$ and blowing it down.\\

The boundary map on manifolds induces a functor:
\begin{align}
\partial: \CK &\rarr \Phi \nonumber \\
\delta e_1 \otimes \cdots \otimes \delta e_n &\mapsto e_1 \otimes \cdots \otimes e_n \nonumber \\
M_L &\mapsto \partial M_L = M
\end{align}
where $M$ is an orientable 3-manifold. We can extend this functor to sequences:
\beq
\partial(\cdots \rarr M_{L_1} \rarr \cdots \rarr M_{L_n} \rarr \cdots ) := \cdots \rarr \partial M_{L_1} \rarr \cdots \rarr \partial M_{L_n} \rarr \cdots
\eeq

From \cite{R}, it follows that:
\begin{Mod}
Under the boundary functor $\partial: \CK \rarr \Phi$, $\Phi$ appears as a moduli category of Kirby equivalence classes $[M_L]$, $M_L$ a morphism in $\CK$
\end{Mod}

\section{Fields and 4-manifolds} \label{section3}
\subsection{Open questions and constrained manifolds}
What we have worked with in the previous section is smooth 4-manifolds $M_L$ built from gluing 2-handles along framed links $L \subset S^3 = \partial D^4$, which is an example of a handle decomposition (\cite{GS}). To have a handle decomposition a 4-manifold must be smoothable. Not all 4-folds are smoothable however, and correspondingly we have an obstruction theory for that (\cite{H}). We would like to have something akin to handle decompositions in the non-smooth case or at the very least some structure intrinsic to non-smooth manifolds on which we have an $S^1$-action. The reason for even considering non-smoothable 4-manifolds is that we are partly interested in physical fields that have 4-manifolds for target space, hence we should really consider a generalization of $\CK$ whose morphisms are directed families of general 4-manifolds. The reason for wanting such 4-manifolds to have an intrinsic $S^1$-action is that they are being seen as dynamical objects insofar as they are target spaces of dynamical fields. Thus we would like such spaces to be constructed from some geometric structure on which we have an $S^1$-action. The question we are asking ourselves is how far into generalizations can we go if such an $S^1$-action is something we deem to be essential to defining dynamical spaces.\\

In a first time we ask whether for 4-manifolds $M$ we have a surjective map $H_{S^1}^*(LM) \rarr H_{S^1}^*(M)$ so that deformations of $M$ are induced by some $S^1$-action on the loop space $LM$, and for what kinds of 4-manifolds do we have such a map. Jones and Petrak define in \cite{JP} a variant of equivariant cohomology $h_{S^1}^*(X) :=H^*(\Omega_{S^1}^*(X)[[u,u^{-1}]])$, $u$ an indeterminate of degree 2, with differential $d+u\iota$, $d$ the differential on $\Omega^*(X)$, $\iota: \Omega^n(X) \rarr \Omega^{n-1}(X)$ the interior product with the vector field generating the $S^1$ action, a cohomology theory for which we do have an isomorphism $h_{S^1}^*(LM) \rarr h_{S^1}^*(M)$. This is the much sought-after $S^1$-action on an underlying structure: the $S^1$ action on $LM$. This is valid for smooth manifolds however. To have a similar statement for general 4-manifolds is prohibitely difficult to establish.\\

Thus it seems generalizing $\CK$ to some category whose morphisms are differentiable manifolds or even singular 4-manifolds, not simply smooth manifolds, should be obtained not by generalizing some intrinsic $S^1$-action but should be approached differently. Whether it be singularities or $C^k$-differentiable structures for example, those we regard as additional constraints that are brought in by hand. They constitute a collection $\fC$ of \textbf{constraints}. One can regard constraints as graded families of properties $\{\mathbf{P}^{(i)} \}_{i \in I}$, $I$ a subset of the integers with 1 as smallest element, for which for all $i \in I$ we have a well-defined map or transition from $\mathbf{P}^{(i)}$ to $\mathbf{P}^{(i+1)}$, and constraints are transitions $\mathbf{P}^{(i)} \rarr \mathbf{P}^{(j)}$ that result from a composition of transitions within a given family, $\mathbf{P}^{(1)}$ being a property of smooth 4-manifolds within such a family. This leads us to defining collections of 4-manifolds $M^{4, \mathbf{P}(i)}$ that share the same property $\mathbf{P}(i)$ which we denote by $\CM^{4, \mathbf{P}(i)}$. Such 4-manifolds with a given property $\mathbf{P}(i)$ other than $\mathbf{P}^{(1)}$ will be referred to as \textbf{constrained 4-manifolds}. This means that ultimately if we impose a constraint $\mathbf{P}$ on smooth 4-manifolds then we can find a family of constraints originating at $\CM^4$ with a property $\mathbf{P}^{(1)}$ for which $\mathbf{P} = \mathbf{P}^{(i)}$ for some $i>1$ and the map on collections of 4-manifolds induced by the transition $\mathbf{P}^{(1)} \rarr
\mathbf{P}^{(i)}$ we denote by $\CM^4 \xrightarrow{d[\mathbf{P}^{(i)}]} \CM^{4,\mathbf{P}^{(i)}}$.\\

This constraint picture however abandons the concept of having an intrinsic $S^1$ action on the constrained manifolds as pointed out above. In most cases an $S^1$ action on smooth 4-manifolds from which constrained 4-manifolds are obtained will not survive the transitions of constraints. Thus for constrained manifolds we have no satisfactory notion of an intrinsic $S^1$ action, a manifestation of the fact that constrained 4-manifolds are still dynamical objects.\\

If we denote by an embedding map the concept of having a physical theory corresponding to a Lagrangian density $\CL$ living on an object $M$ of $\CM^{4,\mathbf{P}^{(i)}}$, then for families of constraints an interesting question which we will not pursue here is whether we have a pullback $d[\mathbf{P}^{(i)}]^*$ such that the following diagram commutes:
\beq
\setlength{\unitlength}{0.5cm}
\begin{picture}(10,9)(0,0)
\thicklines
\put(0,1){$d[\mathbf{P}^{(i)}]^* \CL$}
\put(9.5,1){$\CL$}
\put(1,7){$\CM^{4,\mathbf{P}}$}
\put(9,7){$\CM^{4,\mathbf{P}^{(i)}}$}
\put(4,7){\vector(1,0){4}}
\put(5,8){$d[\mathbf{P}^{(i)}]$}
\put(9.5,3){$\cup$}
\put(10,3){\vector(0,1){3}}
\multiput(8,1.2)(-0.3,0){14}{\line(1,0){0.2}}
\put(4.2,1.2){\vector(-1,0){0.3}}
\put(1.5,3){$\cup$}
\put(2,3){\vector(0,1){3}}
\end{picture}
\eeq

Thus at this point our generalization of $\CK$ will be denoted by $\CK^{\fC}$, its objects are thickened events $\delta e_i$, the same as those of $\CK$. However morphisms will be constrained 4-manifolds. Since ultimately we have physical theories defined on such spaces we enhance $\CK^{\fC}$ to $\CK^{\bCL, \fC}$, keeping the same objects but with morphisms now being pairs $(\CL, M^4)$, $\CL$ a Lagrangian density for a physical theory defined on $M^4$ a constrained 4-manifold. In a sense we have a forgetful functor $\CK^{\bCL, \fC} \rarr \CK^{\fC}$ exhibiting $\CK^{\fC}$ as a moduli category. If we regard Lagrangian densities as being deformations, we aim at establishing the existence of a functor $\mathbb{S}^{\infty}[\Lambda] \leftarrow \CK^{\bCL, \fC}$, which will be done elsewhere.\\

\subsection{Physical theories and Derived Algebraic Geometry}
Focusing our attention on the Lagrangian densities, dynamics in physical theories is derived from equations of motion which themselves are obtained from a Lagrangian. Recall that for a space $X$, a smooth manifold of dimension $n$, $E \rarr X$ a bundle in which fields take their values with corresponding jet bundle $j_{\infty}E \rarr X$, $\Omega^{\cdot, \cdot}(j_{\infty}E)$ the corresponding variational bicomplex, a Lagrangian with values in $E$ is an element $L \in \Omega^{n,0}(j_{\infty}E)$. The variational bicomplex is essentially the deRham complex of $j_{\infty}E$ with differential forms bigraded by horizontal derivations with respect to $X$ and vertical derivations along the fibers of $j_{\infty}E$ as $D=d + \delta$ if $d$ is the deRham differential and $\delta$ is the variational differential (\cite{T}, \cite{D}). It is really the Lagrangian density $\CL \in j_{\infty}E$ however that we are looking at, and it would be preferable to recast this in the formalism of D-modules.\\

We could have worked with D-modules over schemes but those are not well-behaved with respect to quotients so it is natural instead to consider algebraic stacks. Having obstruction problems in mind it is actually better to work with derived algebraic stacks which are built from simplicial rings. Also from \cite{To} derived stacks come into play for obstruction problems and those we encounter under the form of constraints in 4-manifold theory. From \cite{To} for $k$ a commutative ring, the Segal category of derived stacks is denoted $dSt(k)$, it's obtained as the homotopy theory of the model category of derived stacks [HAG II]. The main point is that we have a well-defined theory of D-modules on derived algebraic stacks (\cite{DG}).\\

We consider D-modules over derived algebraic stacks given by a map:
\beq
\mathcal{D}: dSt(k)^{op} \rarr Cat_{\infty}
\eeq
given on objects by the assignment, for $\mathcal{X} \in dSt(k)$ of $\mathcal{D}(\mathcal{X})$ constructed from a functor:
\begin{align}
\Big( DGSch_{/\mathcal{X}} \Big)^{op} &\rarr Cat_{\infty} \nonumber \\
(Y, f:Y \rarr \mathcal{X}) &\mapsto \mathcal{D}(Y)
\end{align}
via:
\beq
\mathcal{D}(\mathcal{X}) = \lim_{\substack{ \leftarrow \\ (Y,f) \in (DGSch_{/\mathcal{X}})^{op}}} \mathcal{D}(Y)
\eeq
the precise definition of which can be found in \cite{DG}.\\

For $\mathcal{X}$ a derived algebraic stack, $\mathcal{D}(\mathcal{X})$ the category of D-modules over $\mathcal{X}$, $\mathcal{F}$, $\mathcal{G}$ and $\mathcal{H}$ in $\mathcal{D}(\mathcal{X})$, we consider, for illustrative purposes, the Lagrangian density $\CL$ defined as follows:
\beq
\CL=Y \partial \psi + \zeta
\eeq
where $Y \in \Gamma(\mathcal{X},\mathcal{F})$, $\psi \in \Gamma(\mathcal{X},\mathcal{G})$, and $\zeta \in \Gamma(\mathcal{X},\mathcal{H})$. Let $\varphi$ be a morphism from $\mathcal{F}$ to $\mathcal{G}$, $\rho$ a morphism from $\mathcal{G}$ to $\mathcal{H}$. Then $\CL$ should really be written:
\begin{align}
\CL&=\rho \Big( \varphi (Y) \bullet_{\mathcal{G}} \partial \psi \Big) +_{\mathcal{H}} \zeta \nonumber \\
&=\Bigg\{\rho \Big( \varphi  \bullet_{\mathcal{G}} \Big) +_{\mathcal{H}}\Bigg \}(Y,\partial \psi, \zeta)
\end{align}

We have a functor:
\beq
\mathcal{D}(\mathcal{X}) \otimes \mathcal{D}(\mathcal{X}) \rarr \mathcal{D}(\mathcal{X})
\eeq
a composition:
\begin{align}
\mathcal{D}(\mathcal{X}) \otimes \mathcal{D}(\mathcal{X}) &\rarr \mathcal{D}(\mathcal{X}\times \mathcal{X})\xrightarrow{\Delta^{!}} \mathcal{D}(\mathcal{X}) \nonumber \\
\mathcal{F} \otimes \mathcal{G} & ------\dashrightarrow \mathcal{F} \overset{!}\otimes \mathcal{G}
\end{align}
giving $\mathcal{D}(\mathcal{X})$ a structure of a symmetric monoidal structure as observed in (\cite{DG}). We have a preliminary map $preL$ that selects those D-modules $\mathcal{F}$, $\mathcal{G}$ and $\mathcal{H}$ such that:
$\CL \in \Gamma(\mathcal{X}, \mathcal{F} \overset{!} \otimes \mathcal{G} \overset{!} \otimes \mathcal{H})$:

\beq
\setlength{\unitlength}{0.5cm}
\begin{picture}(24,5)(0,0)
\put(0,4){$preL: \mathcal{D}(\mathcal{X})$}
\put(8,4){$\mathcal{D}(\mathcal{X}) \otimes \mathcal{D}(\mathcal{X}) \otimes \mathcal{D}(\mathcal{X})$}
\put(19,4){$\mathcal{F} \otimes \mathcal{G} \otimes \mathcal{H}$}
\put(11.5,0){$\mathcal{D}(\mathcal{X})$}
\put(16,0){$\ni$}
\put(19,0){$\mathcal{F} \overset{!}\otimes \mathcal{G} \overset{!}\otimes \mathcal{H}$}
\put(5,4.2){\vector(1,0){2}}
\put(5,5){$\otimes \Delta ^2$}
\put(16.5,4.2){\vector(1,0){2}}
\multiput(21,3.2)(0,-0.3){7}{\line(0,1){0.2}}
\put(21,1.4){\vector(0,-1){0.3}}
\put(12,3.2){\vector(0,-1){2}}
\end{picture}
\eeq

Then $\CL$ itself can be presented as a tree:
\beq
\setlength{\unitlength}{0.5cm}
\begin{picture}(5,7)(0,0)
\put(-1,5.5){$Y$}
\put(0,5.5){$\overset{!}\otimes$}
\put(1,5.5){$\partial \psi$}
\put(3,5.5){$\overset{!}\otimes$}
\put(4,5.5){$\zeta$}
\put(0,5){\vector(1,-1){1}}
\put(0,4){$\varphi$}
\put(1.5,5){\vector(0,-1){1}}
\put(1,3.5){$\bullet_{\mathcal{G}}$}
\put(2,3){\vector(1,-1){1.5}}
\put(2,2){$\rho$}
\put(4,5){\vector(0,-1){3.5}}
\put(3.5,0.8){$+_{\mathcal{H}}$}
\end{picture}
\eeq
If we denote by L$^{\text{alg}}$ this map, which can be encoded by an operad, then we have:
\beq
\Gamma(\mathcal{X}, \mathcal{D}(\mathcal{X})) \xrightarrow{\text{L}^{\text{alg}} \circ \text{preL}= \CL} \Gamma(\mathcal{X}, \mathcal{D}(\mathcal{X}))
\eeq
Note that preL is induced by $\CL$ and Verdier duality. Indeed $\CL$ is built from elements of:
\beq
\text{Map}_{\mathcal{D}(\mathcal{X})}(\mathbf{D}^{\text{Verdier}}_{\mathcal{X}}(\mathcal{F}), \mathcal{G}) \simeq \Gamma(\mathcal{X}, \mathcal{F} \overset{!}\otimes \mathcal{G})
\eeq
for $\mathcal{F}$ and $\mathcal{G}$ D-modules on $\mathcal{X}$. Let $\varphi_{\text{BV}}:=(,)-\Delta$ where $(,)$ is a bracket on $\Gamma(\mathcal{X}, \mathcal{D}(\mathcal{X}))$ and $\Delta$ is an appropriately chosen BV operator (\cite{S}). We have the sequence:
\beq
\Gamma(\mathcal{X}, \mathcal{D}(\mathcal{X})) \xrightarrow{\mathbf{\CL}} \Gamma(\mathcal{X}, \mathcal{D}(\mathcal{X})) \xrightarrow{\varphi_{\text{BV}}} \Gamma(\mathcal{X}, \mathcal{D}(\mathcal{X}))
\eeq
and $Ker \varphi_{\text{BV}}$ is given by those maps $\mathbf{\CL}$ that correspond to a physical Lagrangian.\\

\section{$\CRi^+$, or the $\infty$-category of $\infty$-ribbons} \label{section4}
Braids close to links. Thus it is natural to ask what are morphisms in braided monoidal $\infty$-categories closing to. Further since we will project the corresponding algebra to the level of the Kirby category, which involves framed links, one would like to introduce a notion of framing at the level of braided monoidal $\infty$-categories, whence the need to develop a notion of ribbon $\infty$-categories.\\

In \cite{Lu1}, the observation is being made that for $\CC$ an ordinary category, endowing this latter with the structure of a braided monoidal category is equivalent to endowing its nerve $N(\CC)$ with the structure of an $\mathbb{E}_2$-monoidal $\infty$-category (equiv. an $\mathbb{E}_2$-algebra object of Cat$_{\infty}$). $\mathbb{E}_2$ refers to the little 2-disks operad, and an $\mathbb{E}_2$-algebra object of Cat$_{\infty}$ can equivalently be rephrased as being an $\infty$-algebra over the $\infty$-operad $\mathbb{E}_2$. We will define all these terms shortly. It is thus natural to define, as done in \cite{Lu1}, a braided monoidal $\infty$-category to be an $\mathbb{E}_2$-algebra in Cat$_{\infty}$.\\

In a first time we will introduce (to some extent) all the relevant terminology. We will follow by giving the definition of a braided monoidal $\infty$-category just as it is done in \cite{Lu1}. Finally we will generalize that to the notion of a twisted braided monoidal $\infty$-category.\\

\subsection{Background on $\infty$-operads}
As usual we will denote by $\Gamma$ the Segal category of pointed finite sets.
\begin{iop} (\cite{Lu1})
An \textbf{$\infty$-operad} $\mathcal{O}^{\otimes}$ is defined to be given by a functor $p:\mathcal{O}^{\otimes} \rarr N(\Gamma)$ of $\infty$-categories subject to the following conditions:
\begin{itemize}
\item For any inert morphism (a morphism whose pre-image is a single element) $f: \langle m \rangle \rarr \langle n \rangle$ in $N(\Gamma)$, for any object $x$ in $\mathcal{O}^{\otimes}_{\langle m \rangle}$, there is a $p$-cocartesian morphism $\overline{f}:x \rarr x'$ in $\mathcal{O}^{\otimes}$, lifting $f$, and inducing a functor $f_{!}:\mathcal{O}^{\otimes}_{\langle m \rangle} \rarr \mathcal{O}^{\otimes}_{\langle n \rangle}$
\item For $x$ in $\mathcal{O}^{\otimes}_{\langle m \rangle}$, $x'$ in $\mathcal{O}^{\otimes}_{\langle n \rangle}$, $f: \langle m \rangle \rarr \langle m \rangle$ a morphism in $\Gamma$, Map$^f_{\mathcal{O}^{\otimes}}(x,x')$ the union of the connected components of Map$_{\mathcal{O}^{\otimes}}(x,x')$ over $f$, $p$-cocartesian morphisms $x' \rarr x'_i$ lying over the inert morphisms $\rho_i: \langle n \rangle \rarr \langle 1 \rangle$, $1 \leq i \leq n$, then we have a homotopy equivalence:
    \beq
    \text{Map}^f_{\mathcal{O}^{\otimes}}(x,x') \rarr \prod_{1 \leq i \leq n}\text{Map}^{\rho_i \circ f}_{\mathcal{O}^{\otimes}}(x,x'_i)
    \nonumber
    \eeq
\item For any finite collection of objects $x_1, \cdots, x_n$ of $\mathcal{O}^{\otimes}_{\langle 1 \rangle}$ there is some object $x$ of $\mathcal{O}^{\otimes}_{\langle n \rangle}$ and a collection of $p$-cocartesian morphisms $x \rarr x_i$ covering the morphisms $\rho_i: \langle n \rangle \rarr \langle 1 \rangle$.
\end{itemize}
\end{iop}
\begin{NGamma}
$N(\Gamma)$ is an $\infty$-operad.
\end{NGamma}
\begin{Catinf}
As in \cite{W} let $\CK$ denote the model category of simplicial sets with the Kan model structure, PC$(\CK)$ the category of $\CK$-enriched precategories, a model category of $\infty$-categories. Cat$_{\infty}$ is the localization of PC$(\CK)$. Since this latter turns out to be a symmetric simplicial model category for the cartesian product, it follows that Cat$_{\infty}$ is a symmetric monoidal $\infty$-category (\cite{W}), hence it can be viewed as an $\infty$-operad (\cite{Lu1}).
\end{Catinf}
\begin{mapofops} (\cite{Lu1})\label{iopmaps}
$\mathcal{O}^{\otimes}$ and $\mathcal{O'}^{\otimes}$ being two $\infty$-operads, a \textbf{map of $\infty$-operads} from $\mathcal{O}^{\otimes}$ to $\mathcal{O'}^{\otimes}$ is a map of simplicial sets $f: \mathcal{O}^{\otimes} \rarr \mathcal{O'}^{\otimes}$ such that the following diagram commutes:
\beq
\setlength{\unitlength}{0.5cm}
\begin{picture}(8,5)(0,0)
\put(0,3){$\mathcal{O}^{\otimes}$}
\put(2,3){\vector(1,0){3}}
\put(3,3.5){$f$}
\put(6,3){$\mathcal{O'}^{\otimes}$}
\put(1,2.5){\vector(1,-1){1.7}}
\put(3,0){$N(\Gamma)$}
\put(5.8,2.5){\vector(-1,-2){0.8}}
\end{picture}
\nonumber
\eeq
and $f$ carries inert morphisms in $\mathcal{O}^{\otimes}$, that is $p$-cocartesian morphisms projecting down to inert morphisms in $N(\Gamma)$, to inert morphisms in $\mathcal{O'}^{\otimes}$.
\end{mapofops}
Fibrations of $\infty$-operads are maps between $\infty$-operads that in addition are categorical fibrations, which we now define:
\begin{catfib} (\cite{Lu2})
There is a left proper combinatorial model structure on the category of simplicial sets for which cofibrations are monomorphisms and categorical equivalences are maps of simplicial sets $S \rarr S'$ inducing simplicial functors $\fC[S] \rarr \fC[S']$ that are equivalences of simplicial categories, under which a map of simplicial sets is called a \textbf{categorical fibration} if it has the right lifting property with respect to cofibrations and categorical equivalences.
\end{catfib}
\begin{fibofop} (\cite{Lu1})
A map of $\infty$-operads is called a \textbf{fibration of $\infty$-operads} if it is a categorical fibration.
\end{fibofop}

If $\CC^{\otimes} \rarr \mathcal{O}^{\otimes}$ is a fibration of $\infty$-operads, $\mathcal{D}^{\otimes} \rarr \mathcal{O}^{\otimes}$ is a map of $\infty$-operads, a $\mathcal{D}^{\otimes}$-algebra in $\CC^{\otimes}$ is a map of $\infty$-operads from $\mathcal{D}^{\otimes}$ to $\CC^{\otimes}$ over $\mathcal{O}^{\otimes}$ (\cite{Lu1}). For us $\mathcal{O}^{\otimes}$ will be N$(\Gamma)$, in which case $\mathcal{D}^{\otimes}$-algebra objects of $\CC^{\otimes}$ are given by the $\infty$-category Alg$_{\mathcal{D}^{\otimes}}(\CC^{\otimes})$ of $\infty$-operad maps of \ref{iopmaps}. In particular a braided monoidal $\infty$-category is an E$_2$-algebra object in Cat$_{\infty}$ (\cite{Lu1}).\\

In a subsequent subsection we will make use of a result of \cite{Lu1} that states that the operadic nerve of a fibrant simplicial colored operad is an $\infty$-operad. We define these terms presently.\\

\begin{coloredop}(\cite{Lu1})
For a precise definition the reader is refered to the reference. For our purposes it suffices to note that a \textbf{colored operad} $\mathcal{O}$ consists of a collection $\{X, Y, Z, \cdots\}$ that is considered to be a collection of objects (or colors) of $\mathcal{O}$. For any indexing set $I$, for any collection of objects $\{X_i\}_{i \in I}$ of $\mathcal{O}$, $Y$ an object of $\mathcal{O}$, then we denote by Mul$_{\mathcal{O}}(\{X_i\}_{i \in I}, Y)$ the set of morphisms from the collection $\{X_i\}_{i \in I}$ to $Y$. We have a well-defined notion of composition, for which we have a collection of morphisms $\{id_X \in \text{Mul}_{\mathcal{O}}(\{X\}, X)\}_{X \in \mathcal{O}}$ that is both a left and a right unit. We have associativity of the composition.
\end{coloredop}
A colored operad becomes a \textbf{simplicial colored operad} upon replacing the morphism sets Mul$_{\mathcal{O}}(\{X_i\}_{i \in I}, Y)$ by simplicial sets. If we still denote by $\mathcal{O}$ the resulting simplicial colored operad, we can define an intermediate simplicial  category $\mathcal{O}^{\otimes}$ by taking as objects pairs $X=(\langle n \rangle, (X_1, \cdots, X_n))$, $X_1, \cdots, X_n$ being colors of $\mathcal{O}$, and for another object $Y=(\langle m \rangle, (Y_1, \cdots, Y_m))$ of $\mathcal{O}^{\otimes}$ to define a simplicial set Map$_{\mathcal{O}^{\otimes}}(X,Y)$ to be:
\beq
\coprod_{f:\langle n \rangle \rarr \langle m \rangle} \prod_{1 \leq i \leq m}\text{Mul}_{\mathcal{O}}(\{X_j\}_{f(j)=i}, Y_i)
\nonumber
\eeq
Composition is obvious. Once $\mathcal{O}^{\otimes}$ is defined, we call the simplicial nerve N$^{\otimes}(\mathcal{O})$ of $\mathcal{O}^{\otimes}$ the \textbf{operadic nerve} of $\mathcal{O}$. Once we have a simplicial colored operad $\mathcal{O}$, it is further said to be a \textbf{fibrant simplicial colored operad} if the morphism sets Mul$_{\mathcal{O}}(\{X_i\}_{i \in I}, Y)$ are fibrant, that is Kan complexes.

\subsection{The twisted little k-disks operad $X\mathbb{E}_k$}
This subsection has for aim to define a slight generalization of the little $k$-disks operad as defined in \cite{Lu1}. The reader is referred to that reference for the original definition. Our definition is exactly the same as that of Lurie with the addition that integers $i$ are doubled via the shadow map. Thus we regard the elements $\delta i$ as ribbons. We can twist those ribbons, whence the notion of a framing, which gives the operad the name ``twisted". We have stayed away from the notion of ``framed little $k$-disks operad" as we find it too rigid for our purposes.\\

We generalize the rectilinear embeddings of \cite{Lu1} to twisted rectilinear embeddings by considering graded topological spaces Rect$(\square^k \times S, \square^k)$. Let $S$ be a finite set. An open embedding $\square^k \times S \hookrightarrow \square^k$ whose restriction to each connected component of $\square^k \times S$ is rectilinear in the sense of \cite{Lu1} and twists the elements of $S$ by some vector $\pi \overrightarrow{n}$ for some $\overrightarrow{n} \in \mathbb{Z}^{k-1}$ is said to be a \textbf{twisted rectilinear embedding}. One can easily see that our definition is a simple generalization of the rectilinear embedding as found in \cite{Lu1}:
\begin{align}
f:\square^k \times \delta i &\rarr \square^k \nonumber \\
((x_1, \cdots, x_k), \delta i) &\mapsto ((a_1x_1 + b_1, \cdots, a_kx_k+b_k), R_{\pi \overrightarrow{n}}\delta i)
\end{align}
We denote by Rect$^{\times}(\square^k \times S, \square^k)$ the collection of all twisted rectilinear embeddings from $\square^k \times S$ to $\square^k$ that we regard as a topological space by giving $\mathbb{Z}$ the discrete topology and by giving $\mathbb{R}^{2k} \times \mathbb{Z}^{k-1}$ the product topology. We identify Rect$^{\times}(\square^k \times S, \square^k)$ with an open subset of $(\mathbb{R}^{2k} \times \mathbb{Z}^{k-1})^S$. Having topologized this set of twisted embeddings we can now define a topological operad whose $n$-ary operations are parametrized by Rect$^{\times}(\square^k \times \{ \delta 1, \cdots, \delta n\})$.

\begin{tXEk}
We define a topological category $\:^tX\mathbb{E}_k$ as follows: objects are $\delta \langle n \rangle$'s where the $\langle n \rangle$'s are objects of $\Gamma$. For $\langle m \rangle$, $\langle n \rangle$ objects of $\Gamma$, a morphism from $\delta \langle m \rangle$ to $\delta \langle n \rangle$ in $\:^tX\mathbb{E}_k$ consists of a morphism $\alpha: \langle m \rangle \rarr \langle n \rangle$ in $\Gamma$, and for any $j \in \langle n \rangle^{\circ}$, a twisted rectilinear embedding $\square^k \times \delta
\alpha^{-1}\{j\} \rarr \square^k$. Further, for every $\delta \langle m \rangle$, $\delta \langle n \rangle $ in $\:^tX\mathbb{E}_k$, we have a topology on Hom$_{\:^tX\mathbb{E}_k}(\delta \langle m \rangle, \delta \langle n \rangle)$ induced from the presentation (\cite{Lu1}):
\beq
\coprod_{f:\langle m \rangle \rarr \langle n \rangle} \prod_{1 \leq j \leq n} \text{Rect}^{\times}(\square^k \times \delta f^{-1}\{j\}, \square^k)
\nonumber
\eeq
Composition of morphisms on strictly rectilinear embeddings is the same as the one on $\:^t\mathbb{E}_k$ (see \cite{Lu1}), and it is additive on rotations, componentwise.
\end{tXEk}
If we write:
\beq
X\mathbb{E}_k= N(\:^tX\mathbb{E}_k) \nonumber
\eeq
then $X \mathbb{E}_k$ becomes an $\infty$-category. Further:
\begin{XEk}
$X\mathbb{E}_k$ is an $\infty$-operad
\end{XEk}
\begin{proof}
The proof follows the same lines as those for proving that the little $k$-disks operad is an $\infty$-operad as done in \cite{Lu1}. In our case it suffices to adapt to the twisted case: if we let $\mathcal{O}$ be the simplicial colored operad with a single object $\square^k$ and Mul$_{\mathcal{O}}(\{\square^k\}_{i \in I}, \square^k)$ to be given by Sing Rect$^{\times}(\square^k \times I, \square^k)$, a singular complex of a topological space, then naturally $\mathcal{O}$ is a fibrant simplicial colored operad, hence $X \mathbb{E}_k$ is an $\infty$-operad by Proposition 2.1.1.27 of \cite{Lu1} which states that the operadic nerve of a fibrant simplicial colored operad is an $\infty$-operad.
\end{proof}

\subsection{Ribbon $\infty$-categories}
We are especially interested in $X \mathbb{E}_2$, as $X \mathbb{E}_2$-algebras in Cat$_{\infty}$ would correspond to a notion of braided monoidal $\infty$-categories with a framing, hence \textbf{ribbon $\infty$-categories}. We fix a map $\psi$ from $X \mathbb{E}_2$ to Cat$_{\infty}$ over N$(\Gamma)$ ( the collection of all such maps is Alg$_{X\mathbb{E}_2}(\text{Cat}_{\infty})$, an $\infty$-category (\cite{Lu1})). We denote by $\CRi$ the resulting ribbon $\infty$-category. For $A$ and $B$ two objects of $\CRi$, $f:A \rarr B$ a morphism in $\CRi$, we denote by Map$_{\CRi}(A,B)_f$ the smallest $\infty$-groupoid that contains all morphisms that have $f$ as either domain or codomain. We refer to Map$_{\CRi}(A,B)_f$ as a \textbf{framed $\infty$-braid} $f$. If we were to work with Alg$_{\mathbb{E}_2}(\text{Cat}_{\infty})$, $\psi$ an object of this $\infty$-category corresponding to a braided monoidal $\infty$-category $\mathcal{C}$, $f:A \rarr B$ a morphism in $\mathcal{C}$, we would simply call Map$_{\mathcal{C}}(A,B)_f$ an $\infty$-braid $f$. Now if $A \sim B$, one can identify those objects. This would mean closing the geometric braid $f$. Closing $f$ into $\overline{f}$ induces a quotienting map:
\beq
\text{Map}_{\CRi}(A,B)_f \rarr \text{Map}_{\CRi}(A,B)_{\overline{f}}
\eeq
that is defined inductively and yields a framed $\infty$-link $\text{Map}_{\CRi}(A,B)_{\overline{f}}$. There are objects $a$ and $b$ of $X \mathbb{E}_2$ such that $\psi(a)=A$ and $\psi(b)=B$. Further we can find a morphism $\tilde{f}$ in Map$_{X \mathbb{E}_2}(a,b)$ such that $\psi(\tilde{f})=f$. Closing means we collapse domain and codomain of a given morphism. Let $\pi$ be such a map on objects. Identifying $a$ and $b$ is done by:
\beq
\pi_{a,b}: a,b \rarr a \sim b
\eeq
which induces a map $\Pi_{a,b}$ on morphisms. We are led to defining $\Pi_{a,b}X \mathbb{E}_2$ to be $X \mathbb{E}_2$ where $a$ and $b$ have collapsed as well as all relevant higher morphisms. We define:
\beq
\text{Map}_{\CRi}(A,B)_{\overline{f}} = \Big(
\psi \text{Map}_{\Pi_{\psi^{-1}A, \psi^{-1}B} X \mathbb{E}_2}(\psi^{-1}A, \psi^{-1}B)
\Big)_{\overline{f}}
\eeq
where:
\beq
\overline{f}= \psi \Pi_{\psi^{-1}A, \psi^{-1}B} \psi^{-1}f
\eeq

We are considering the particular map $\psi$ that corresponds to the ribbon $\infty$-category $\CRi$ whose objects are thickened events $\delta e_i$, $\delta$ the shadow map, $e_i$ an event, some of whose morphisms are braids:
\beq
\delta e_1 \otimes \cdots \otimes \delta e_n \xrightarrow{f} \delta e_1 \otimes \cdots \otimes \delta e_n
\eeq
with closures $\overline{f}$ with a corresponding framed $\infty$-link:
\beq
\text{Map}_{\CRi}(\otimes^n \delta e_i, \otimes^n \delta e_i)_{\overline{f}}
\eeq
The same objects, braids and links are also those of $\CK$, with the link $\overline{f}$ giving rise to a handlebody $M_{\overline{f}}$. At the level of $\CRi$ this corresponds to introducing \textbf{$\infty$-4-manifolds} which we now define. We can define an $\infty$-4-manifold to be a pair:
\beq
(M_{\overline{f}}, \text{Map}_{\CRi}(\otimes ^n \delta e_i, \otimes ^n \delta e_i)_{\overline{f}})
\eeq
element of:
\beq
^{\infty}\CM^4 := \CM^4 \times_{\text{Mor}(\CK)} \text{Map}_{\CRi}
(\text{Ob}(\CRi),\text{Ob}(\CRi))_{\overline{\text{Mor}(\CRi)}}
\eeq
where the projection on the second factor yields $\overline{f}$, projecting down to $M_{\overline{f}}$ in Mor$(\CK)$. This gives rise to the enhanced category $\CRi^+$ whose objects are directed families of elements of $^{\infty}\CM^4$ with composition induced by concatenation of directed families with identity for the composition being $\Big(  \natural n S^2 \times D^2, \text{Map}_{\CRi}(\otimes^n \delta  e_i, \otimes^n \delta e_i) _{\overline{id}}\Big)$. In this case we have a projection from $\CRi^+$ down to $\CK$:
\begin{align}
\CRi^+ & \rarr \CK \nonumber \\
\rarr_{j \in J}(M_{\overline{f}_j},\text{Map}_{\CRi}(\otimes ^n \delta e_i, \otimes^n \delta e_i)_{\overline{f}_j}) & \mapsto \,\,\Big(\rarr_{j \in J}M_{\overline{f}_j}\Big)
\end{align}
obtained by restriction to those morphisms of $\CRi^+$ that only involve $\infty$-4-folds.

\subsection{Walled braids}
It is natural to close braids $f:A \rarr B$ to links if $A=B$. This is not difficult if $A=B=\otimes^n \delta e_i$. For braided monoidal $\infty$-categories on which we have a model structure, we can extend the closing of braids to those weak equivalences $f: \otimes_{i \in I,\; |I|=n} \delta e_i \rarr \otimes_{j \in J,\; |J|=n} \delta e_j$ where $I \neq J$. In particular if we can put a model structure on $\CRi$ then one could do just that. In general closing braids is not natural. Thus a morphism in a braided monoidal $\infty$-category:
\beq
f: A_1 \otimes \cdots \otimes A_n \rarr B_1 \otimes \cdots \otimes B_n
\eeq
may not necessarily close. Such a morphism we picture as a geometric braid stretching between two walls, one having objects $A_1, \cdots, A_n$ at a given height, separated by a same distance, the other wall having the objects $B_1, \cdots, B_n$ at the same height, separated by a same distance. We regard such a geometric braid along with its two accompanying walls as being embedded in $S^3$ and refer to those objects as \textbf{walled braids}. If those braids are actually ribbons we glue copies of $D^2 \times \widetilde{D^2}$ along those strands in accordance with the framing on those braids, where $\widetilde{D^2}$ is the blow up of $D^2$ at the origin. In the event that the domain and codomain of such a morphism are identified, we blow down $\widetilde{D^2}$ and identify the two walls which amounts to closing the ribbon into a framed link and the gluing just described amounts to gluing a 2-handle along framed links in $S^3$. Walled ribbons on which we have glued copies of $D^2 \times \widetilde{D^2}$ we call \textbf{walled handlebodies}. In particular objects of $\CRi$ that are not equivalent give rise to such walled 4-manifolds if a morphism between them exist in $\CRi$. Thus in this formalism $\AXECi$ yields a trove of ribbon $\infty$-categories which lead to walled generalizations of the notions introduced in previous sections: walled $\infty$-braids, walled $\infty$-4-manifolds, ...\\

Specifically for $\psi \in \AXECi$, leading to a ribbon $\infty$-category $\CRi$, $A_1 \otimes \cdots \otimes A_n$, $B_1 \otimes \cdots \otimes B_n$ two objects of $\CRi$, $f:\otimes ^n A_i \rarr \otimes^n B_i$ a morphism regarded as a geometric braid in between two walls, Map$_{\CRi}(\otimes ^n A_i, \otimes^n B_i)_f$ we refer to as a \textbf{walled $\infty$-braid}, $f$ is represented by a walled geometric ribbon on which we can glue copies of $D^2 \times \widetilde{D^2}$ to give rise to \textbf{walled 4-manifolds} $M_f$ with a corresponding notion of \textbf{walled $\infty$-4-manifolds} $(M_f, \text{Map}_{\CRi}(\otimes^n A_i, \otimes^n B_i)_f)$. Denote by $w\CM^4$ the collection of walled 4-manifolds. The collection of walled $\infty$-4-manifolds we define to be:
\beq
^{\infty}w\CM^4 := w\CM^4 \times_{\text{Mor}(\CK)} \text{Map}_{\CRi}
(\text{Ob}(\CRi),\text{Ob}(\CRi))_{\text{Mor}(\CRi)}
\eeq
giving rise to the enhanced category $w\CRi^+$ whose morphisms are no longer simple ribbons but are elements of $\overrightarrow{^{\infty}w\CM^4}$, the collection of directed families of elements of $^{\infty}w\CM^4$ with composition being induced by the concatenation of families and the identity for composition being denoted $\Big( (S^2 \times \widetilde{D^2})^{\natural}, \text{Map}_{\CRi}(\otimes^n A_i, \otimes^n B_i)_{id}\Big)$ where as usual $(S^2 \times \widetilde{D^2})^{\natural}$ is a variable identity that stands for $\natural n S^2 \times \widetilde{D^2}$ with $n$ depending on the context.\\

Observe that if $f:A_1 \otimes \cdots \otimes A_n \xrightarrow{\sim} B_1 \otimes \cdots \otimes B_n$ is an equivalence or a weak equivalence in the context of model categories, walls of the walled 4-manifold $M_f$ are identified, this is equivalent to blowing down the copies of $D^2 \times \widetilde{D^2}$ giving rise to $M_{\overline{f}}$. At the level of $\infty$-4-manifolds this corresponds to projecting down:
\beq
(M_f,\text{Map}_{\CRi}(\otimes ^n A_i, \otimes^n B_i)_f) \rarr (M_{\overline{f}},\text{Map}_{\CRi}(\otimes ^n A_i, \otimes^n B_i)_{\overline{f}})
\eeq
In particular if we are focusing our attention on the particular morphism $\psi$ corresponding to the ribbon $\infty$-category of events $\CRi$, it appears more natural to focus instead on $w\CRi^+$ rather than on $\CRi^+$, the objects being the same, with projection:
\beq
\setlength{\unitlength}{0.5cm}
\begin{picture}(19,7)(0,0)
\put(-4,6){$(M_f,\text{Map}_{\CRi}(\otimes ^n \delta e_i, \otimes^n \delta e_i)_f)$}
\put(11,6){$(M_{\overline{f}},\text{Map}_{\CRi}(\otimes ^n \delta e_i, \otimes^n \delta e_i)_{\overline{f}})$}
\put(13,0){$M_{\overline{f}} \in \CK$}
\put(7,6){\vector(1,0){3}}
\put(14,5){\vector(0,-1){3}}
\multiput(4,5)(0.4,-0.2){20}{\circle{0.1}}
\put(11.6,1.2){\vector(2,-1){0.3}}
\end{picture}
\eeq
that generalizes immediately to directed families.\\

Note that walled 4-manifolds are morphisms in $\CK^{\fC}$ as walls correspond to constraints. We have a projection:
\begin{align}
w\CRi^+ & \rarr \CK^{\fC} \nonumber \\
\rarr_{j \in J}(M_{f_j},\text{Map}_{\CRi}(\otimes ^n \delta e_i, \otimes^n \delta e_i)_{f_j}) & \mapsto \,\, \Big(\rarr_{j \in J}M_{f_j} \Big)
\end{align}

\section{Relations between $\Lambda$, $\Phi$, $\CK$, $\CK^{\fC}$, $\CK^{\CL,\fC}$, $\CRi$, $\CRi^+$ and $w\CRi^+$} \label{section5}
What we have so far is the following:
\beq
\setlength{\unitlength}{0.5cm}
\begin{picture}(10,11)(0,-1)
\put(0,-0.3){$\Phi$}
\put(3,0){\vector(-1,0){2}}
\put(2,0.5){$\partial$}
\put(0.3,3){\vector(0,-1){2}}
\put(0,3.7){$\Lambda$}
\put(4,-0.3){$\CK$}
\put(4.3,3){\vector(0,-1){2}}
\put(4,3.7){$\CK^{\fC}$}
\put(8,0){\vector(-1,0){3}}
\put(8.3,-0.3){$\CRi^+$}
\put(8.3,3.7){$w\CRi^+$}
\put(8.6,3){\vector(0,-1){2}}
\put(8,4){\vector(-1,0){2.8}}
\put(1,4){\vector(1,0){2}}
\put(4.3,7){\vector(0,-1){2}}
\put(4,7.7){$\CK^{\CL, \fC}$}
\put(1.2,5){\vector(1,1){2.5}}
\put(11,6.5){$\CRi$}
\put(11.2,6){\vector(-1,-4){1.5}}
\put(11,6){\vector(-1,-1){1.5}}
\put(4.6,5.2){\vector(4,1){6}}
\put(4,5){\line(-4,-1){2.5}}
\end{picture}
\eeq

In the above diagram, the functors $\CRi \rarr w\CRi^+$ and $\CRi \rarr \CRi^+$ produce walled $\infty$-4-manifolds and $\infty$-4-manifolds respectively from morphisms in $\CRi$, extended to directed families. The functors $\CK^{\fC} \rarr \CK$ and $w\CRi^+ \rarr \CRi^+$ are both induced by closing braids into links if possible. Functors $w\CRi^+ \rarr \CK^{\fC}$ and $\CRi^+ \rarr \CK$ are induced by projection on the first component. Functors $\Lambda \rarr \CK^{\CL, \fC}$, $\Lambda \rarr \CK^{\fC}$ and $\Lambda \rarr \CRi$ exist by assumption. They are essentially surjective functors.\\

Now a few remarks are in order. It is more reasonable to assume that $\Lambda$ can be recovered from $\CRi$ rather than assuming it can be recovered from the knowledge of Lagrangian densities of physical theories defined on 4-manifolds. Since $\CRi$ contains all morphisms between events, it would seem $\CRi$ is closer than $\CK^{\CL, \fC}$ in giving information about $\Lambda$. Noting the importance of $\CRi$ in determining $\Lambda$, one natural question is what are maps in $\AXECi$ other than $\psi$ giving rise to $\CRi$ corresponding to relative to $\Lambda$. This question is even more important if one can fully recover $\Lambda$ from the knowledge of $\CRi$. One would need to understand what are those maps in $\AXECi$ that are different from $\psi$, if any.\\

The above paragraph points at some interesting detour. It would be natural to choose $\CRi$ over $\CK^{\CL, \fC}$ to investigate natural phenomena. Note that giving $\CRi$ is equivalent to giving $w\CRi^+$, or equivalently walled $\infty$-4-manifolds, most of their information being encoded in $\infty$-braids which are none other than $\infty$-groupoids. Hence it would seem $\infty$-categories supersede Lagrangian field theories in investigating natural phenomena, something that one could see happening as far back as the work of Kontsevich on the Homological Mirror Symmetry, a mathematical treatise of the well-known Mirror Symmetry problem in Physics (\cite{K}). This importance of higher category theory over Lagrangian field theories in investigating deep questions regarding natural phenomena we regard as the most important point made by the present work.

\newpage

\section*{Appendix}
This appendix has for sole purpose to introduce an abstract notion of a variable in a domain to make a distinction between a variable in a domain and elements thereof.\\

We define \textbf{motivic frames}, an abstraction of some constructs in Mathematics which is motivic in spirit. A motivic frame is a collection of graded objects $\{x^{(n)}\}_{n \geq 0}$, with one object per degree, each object $x^{(i)}$ being built from $x^{(i-1)}$, with maps $x^{(i)} \xrightarrow{z^{(i)}} x^{(i+1)}$ indicating in what manner such a construction is done. For $X$ a mathematical construct we denote by MotFr[$X$] a corresponding motivic frame, and we denote $x^{(n)}[X]$ by MotFr$_n[X]$.\\

\begin{CWCx}
For a CW complex, $x^{(n)}$ is a $n$-skeleton, $z^{(n)}:x^{(n)} \mapsto x^{(n+1)}$ is the process of gluing a disk $D^{n+1}$ along an attaching map $j^{(n)}:S^n \rarr x^{(n)}$.
\end{CWCx}

\begin{Int}
integration has a motivic frame with two objects $x^{(0)}$ being an integrable function, $x^{(1)}$ the integral itself, $z^{(0)}:x^{(0)} \mapsto x^{(1)}$ is the integration.
\end{Int}

\begin{rgsp}
A space $X$ has a motivic frame with 3 objects, $x^{(0)}$ is a point, $x^{(1)}$ is a set of points, $x^{(2)}$ is a topological space, $z^{(0)}:x^{(0)} \mapsto x^{(1)}$ is the collection of points into a set, $z^{(1)}:x^{(1)} \mapsto x^{(2)}$ puts a topology on $x^{(1)}$.
\end{rgsp}

\begin{var}
A variable $x$ in some set $S$ represents an element of that set, hence $x=$ MotFr$_0[S]$.
\end{var}

Note that it is fairly clear from that formalism that for a given mathematical construct there may be different ways to implement such a construction, hence different motivic frames. When we write MotFr$[X]$ we will always specify what motivic frame we are using. Motivic frames for us have the advantage to make a distinction between points and variables representing such points. For instance writing $2 \in [0,3]$ or $x \in [0,3]$, which puts points and variables on a same footing, can be made more precise by saying $x=$ MotFr$_0( [0,3] )$ with $[0,3]=$ MotFr$_1( [0,3] )$ and $z^{(0)}: x \mapsto [0,3]$ amounts to collecting points into a closed interval $[0,3]$. $x$ has the nature of a point in $[0,3]$, whereas a given point in that interval is a specialization of $x$ to the location in $[0,3]$ corresponding to that point.

\end{document}